\documentclass{amsart}

\makeatletter
\def\subsection{\@startsection{subsection}{2}%
  \z@{.5\linespacing\@plus.7\linespacing}{.5\linespacing}%
  {\normalfont\bfseries}}
\makeatother

\makeatletter
\def\@defaultbiblabelstyle#1{[#1]}
\makeatother

\makeatletter
\def\@setauthors{%
  \begingroup
  \def\thanks{\protect\thanks@warning}%
  \trivlist
  \centering\footnotesize \@topsep30\p@\relax
  \advance\@topsep by -\baselineskip
  \item\relax
  \author@andify\authors
  \def\\{\protect\linebreak}%
  \authors%
  \ifx\@empty\contribs
  \else
    ,\penalty-3 \space \@setcontribs
    \@closetoccontribs
  \fi
  \endtrivlist
  \endgroup
}
\def\@settitle{\begin{center}%
  \baselineskip14\p@\relax
    \bfseries
  \@title
  \end{center}%
}
\makeatother

\usepackage{amssymb}
\usepackage{upgreek}
\usepackage{dsfont}
\usepackage{mathdots}
\usepackage[colorlinks=true, citecolor=blue, backref=page]{hyperref}
\renewcommand*{\backrefalt}[4]{%
  \ifcase #1 %
  \else
    {↑#2}%
  \fi}
\usepackage{upref}
\usepackage{tikz}
\usetikzlibrary{calc}
\usepackage{extarrows}
\usepackage[all]{xy}

\newtheorem{theorem}{Theorem}[section]
\newtheorem{lemma}[theorem]{Lemma}
\newtheorem{proposition}[theorem]{Proposition}
\newtheorem{corollary}[theorem]{Corollary}

\theoremstyle{definition}
\newtheorem{definition}[theorem]{Definition}
\newtheorem{example}[theorem]{Example}

\newtheorem{notation}[theorem]{Notation}
\newtheorem{convention}[theorem]{Convention}

\theoremstyle{remark}
\newtheorem{remark}[theorem]{Remark}

\numberwithin{equation}{section}
\allowdisplaybreaks[4]

\newcommand\bC{\mathbb{C}}
\newcommand\bN{\mathbb{N}}
\newcommand\bZ{\mathbb{Z}}
\newcommand\cC{\mathcal{C}}
\newcommand\cH{\mathcal{H}}
\newcommand\cJ{\mathcal{J}}
\newcommand\sfA{\mathsf{A}}
\newcommand\nT{\widetilde{T}}
\newcommand\tld{\widetilde}
\newcommand\lrle{\underset {\mathrm{LR}} {\leqslant}}
\newcommand\lrsim{\underset{\mathrm{LR}}{\sim}}
\newcommand\af{\boldsymbol{a}}
\newcommand\hq[1][1]{q^{ \frac{#1}{2}}}
\newcommand\mhq[1][1]{q^{- \frac{#1}{2}}}
\newcommand\AC{{\mathbb{C}[q^{\pm \frac{1}{2}}]}}
\newcommand\II{\mathds{1}}
\newcommand\abs[1]{\lvert #1 \rvert}

\begin{document}

\title[Representations of \texorpdfstring{$\af$}{a}-function value 1]{Representations of Coxeter groups of Lusztig's \texorpdfstring{$\af$}{a}-function value 1}

\author{Hongsheng Hu}
\address{School of Mathematics, Hunan University, Changsha 410082, China}
\email{huhongsheng(at)hnu(dot)edu(dot)cn}



\date{\today}

\begin{abstract}
In this paper, we give a characterization of Coxeter group representations of Lusztig's $\af$-function value $1$, then determine all the irreducible such representations for certain simply laced Coxeter groups.
\end{abstract}

\maketitle

\setcounter{tocdepth}{3}
\tableofcontents

\section{Introduction} \label{sec-intro}

In the seminal paper \cite{KL79}, D.\ Kazhdan and G.\ Lusztig introduced the notion of (left, right, two-sided) cells for an arbitrary Coxeter group $(W,S)$.
Each cell gives a representation of the corresponding Hecke algebra (and the Coxeter group, after specialization), called a cell representation.
Cell representations have been proved to be of great significance in representation theory.

If $(W,S)$ is irreducible, that is, the Coxeter graph is connected, there is a two-sided cell, first formulated in \cite{Lusztig83-intrep}, consisting of non-unity elements of $W$ whose reduced expression is unique.
This two-sided cell, which we denote by $\cC_1$, is the second-highest with respect to the cell order.
In another remarkable paper \cite{Lusztig85-cell-i}, Lusztig defined a function $\af$ on $W$, which is a very useful tool for studying cells and representations.
The function $\af$ takes constant value on each two-sided cell.
In particular, the $\af$-function value is $1$ on $\cC_1$.
In fact, it is known that $\cC_1$ is characterized as the set of elements with $\af$-function value $1$ (see Lemma \ref{lem-C1}).

It is a long-standing conjecture that the function $\af$ is uniformly bounded on $W$ whenever $(W,S)$ is of finite rank, that is, $\abs{S} < \infty$, and an explicit upper bound $N(W)$ is proposed (see \cite[Question 1.13(iv)]{Xi94} and \cite[Conjecture 13.4]{Lusztig14-hecke-unequal}).
Recently, the boundedness of the function $\af$ is proved in \cite{Chen25} (with a bound possibly greater than $N(W)$).
Given the boundedness property, it is well known that any irreducible representation of $W$ (or of the corresponding Hecke algebra) is a quotient of some cell representation provided by a unique two-sided cell (see Proposition \ref{prop-cell-rep-m}).
As a byproduct, we prove the existence of the lowest two-sided cell in $W$ (see Corollary \ref{cor-cell-low}, and also \cite[Theorem 1.5]{Xi12}), answering the question \cite[Question 1.13(ii)]{Xi94}.

Following Lusztig \cite{Lusztig87-cell-iii}, we may also define an $\af$-function value for an arbitrary representation (see Definition \ref{def-a}).
In particular, if this representation is irreducible, its $\af$-function value equals the $\af$-function value on the corresponding two-sided cell.
The geometric representation of $W$ defined in \cite[Ch.\ V, \S 4]{Bourbaki-Lie456} is of $\af$-function value $1$ (see also \cite[Chapter 12]{Bonnafe} and \cite{Douglass90}).

The main aim of this paper is to study the representations of $W$ of $\af$-function value $1$.
In general, the cell representation provided by the two-sided cell $\cC_1$ is too ``big''.
It possibly admits infinite many non-isomorphic simple quotients (see, for example, Remark \ref{rmk-a=1}\eqref{rmk-a=1-2} and Remark \ref{rmk-simply-laced}\eqref{rmk-simply-laced-1}).
Moreover, it possibly admits simple quotients of infinite dimension as well (see Remark \ref{rmk-a=1}\eqref{rmk-a=1-3}).

The first main result in this paper (Theorem \ref{thm-A1}) is a characterization of representations of $W$ of $\af$-function value $1$.
It states that a nonzero and nontrivial representation of $W$ is of $\af$-function value $1$ if and only if there is no common eigenvector with eigenvalue $-1$ for any pair of defining generators $r,t \in S$ whenever the order $m_{rt}$ of $rt$ is finite.
In particular, all of the generalized geometric representations  in the sense of \cite{Hu23} are of $\af$-function value $1$.
In that paper \cite{Hu23}, all the generalized geometric representations are constructed and classified.
Some other evidences suggest that the $\af$-function value of a representation might be related in some implicit manner to the common eigenvectors with eigenvalue $-1$ of the defining generators (see Remark \ref{rmk-a=1}\eqref{rmk-a=1-5}).

Our second main result (Theorem \ref{thm-A1=R}) is an explicit description and classification of irreducible representations of $\af$-function value $1$ for irreducible simply laced Coxeter groups whose Coxeter graph is a tree or admits only one cycle.
In other words, for these Coxeter groups we are able to find all the simple quotients of the cell representation provided by $\cC_1$.
It turns out that such representations are reflection representations (that is, on the representation space each generator in $S$ acts by an abstract reflection) satisfying some specific conditions.
These representations, called R-representations in  \cite{Hu23}, are closely related to the generalized geometric representations, and they can be described explicitly.
For the sake of convenience, we present briefly the concrete construction and the classification of irreducible R-representations in our specified settings in Subsection \ref{subsec-R}.
In particular, all of these representations are finite dimensional.
However, for other irreducible simply laced Coxeter groups, there exist infinite dimensional irreducible representations of $\af$-function value $1$ (see Remark \ref{rmk-simply-laced}\eqref{rmk-simply-laced-2}).
It seems difficult to classify irreducible representations of $\af$-function value 1 for such $W$.
See also \cite{DPWX22, Xu19} for relevant results for the corresponding based ring $J_{\cC_1}$.

The rest of this paper is organized as follows.
In Section \ref{sec-pre}, we recollect briefly some preliminary knowledge on dihedral groups and Kazhdan-Lusztig cells.
In Section \ref{sec-main-1}, we give the characterization of Coxeter group representations of $\af$-function value $1$.
In Section \ref{sec-main-2}, we focus on certain simply laced Coxeter groups and determine all of their irreducible representations of $\af$-function value $1$.

\subsection*{Acknowledgments}
The author would like to thank Professor Nanhua Xi for his patient guidance and insightful discussions.
The author is also deeply grateful to Professor Si'an Nie for pointing out an error in a previous version of the proof of Theorem \ref{thm-A1=R}.
The author is supported by the Fundamental Research Funds for the Central Universities (no.~531118010972) and the National Natural Science Foundation of China (no.~12501037).

\section{Preliminaries} \label{sec-pre}

We use $e$ to denote the neutral element in a group.
The base field of vector spaces are assumed to be $\bC$ unless otherwise specified.

\subsection{Representations of dihedral groups} \label{subsec-dih}

Let $D_m := \langle r,t \mid r^2 = t^2 = (rt)^m = e \rangle$ be a finite dihedral group ($m \in \bN_{\ge 2}$), and $V_{r,t} := \mathbb{C}\beta_r \oplus \mathbb{C}\beta_t$ be a vector space with formal basis $\{\beta_r, \beta_t\}$.
For natural numbers $k$ satisfying $1 \leq k \leq \frac{m}{2}$, we define the action $\rho_k : D_m \to GL(V_{r,t})$ by
\begin{equation}   \label{eq-Dm-rho-k}
  \begin{aligned}
    r \cdot \beta_r & := - \beta_r, \quad  & r \cdot \beta_t & := \beta_t + 2\cos\frac{k\uppi}{m} \beta_r,\\
    t\cdot \beta_t & := - \beta_t,  & t \cdot \beta_r & := \beta_r + 2\cos\frac{k\uppi}{m} \beta_t.
  \end{aligned}
\end{equation}
Intuitively, $\rho_k$ has a real form spanned by $\beta_r$ and $\beta_t$, and $r,t$ act on the (real) plane by two reflections with respect to two lines with an angle of $\frac{k \uppi}{m}$, see Figure \ref{rho-k}.

\begin{figure}[ht]
    \centering
    \begin{tikzpicture}
      \draw[dashed] (-1.3,0)--(2,0);
      \draw[dashed] (240:1.3)--(60:2);
      \draw[->] (0,0)--(0,1.5);
      \draw[->] (0,0)--(330:1.5);
      \draw (0.3,0) arc (0:60:0.3);
      \draw[<->] ($(1.9,0) + (300:0.2)$) arc (300:420:0.2);
      \draw[<->] ($(60:1.9) + (0.2,0)$) arc (0:120:0.2);
      \node[left] (ar) at (0,1.5) {$\beta_r$};
      \node[below] (at) at (330:1.5) {$\beta_t$};
      \node[right] (r) at (2.1,0) {$r$};
      \node[right] (t) at (63:2.3) {$t$};
      \node[right] (angle) at (0.2,0.3) {$\frac{k\uppi}{m}$};
    \end{tikzpicture}
    \caption{$\rho_k : D_m \to GL(V_{r,t})$}\label{rho-k}
\end{figure}

If $k < \frac{m}{2}$, then $\rho_k$ is irreducible.
If $m$ is even and $k = \frac{m}{2}$, then $\rho_\frac{m}{2} \simeq \varepsilon_r \oplus \varepsilon_t$ splits into a direct sum of two representations of dimension $1$, where
\begin{equation}   \label{eq-epsilon-rt}
    \varepsilon_r: r \mapsto -1, t \mapsto 1; \quad \quad
    \varepsilon_t: r \mapsto 1, t \mapsto -1.
\end{equation}
We denote by $\II$ the trivial representation, and by $\varepsilon$ the sign representation, that is,
\begin{equation}    \label{eq-epsilon}
  \II: r,t \mapsto 1; \quad \quad \varepsilon: r,t \mapsto -1.
\end{equation}

\begin{lemma} [{\cite[\S 5.3]{Serre77}}] \label{lem-Dm-rep}
The following are all the irreducible representations of $D_m$,
\begin{gather*}
    \{\II, \varepsilon\} \cup \{\rho_1, \dots, \rho_{\frac{m-1}{2}}\}, \quad \text{if $m$ is odd}; \\
    \{\II, \varepsilon, \varepsilon_r, \varepsilon_t\} \cup \{\rho_1, \dots, \rho_{\frac{m}{2}-1}\}, \quad \text{if $m$ is even}.
\end{gather*}
These representations are non-isomorphic to each other.
\end{lemma}

\begin{remark} \label{rmk-Dm-rep}
  The $-1$-eigenspace of either $r$ or $t$ in the representation space $V_{r,t}$ of $\rho_k$ $(1 \le k \le \frac{m}{2})$ is one dimensional.
\end{remark}


\subsection{The Kazhdan-Lusztig cells and Lusztig's function \texorpdfstring{$\af$}{a}}

Let $(W,S)$ be a Coxeter group with Coxeter matrix $(m_{st})_{s,t \in S}$ and length function $\ell$.
For simplicity, we assume $|S|$ is finite.
Following the seminal papers \cite{KL79, Lusztig85-cell-i}, the \emph{Hecke algebra} $\mathcal{H}$ of $(W,S)$ is defined as follows.
Let $\hq$ be an indeterminate and $\bZ [q^{\pm \frac{1}{2}}]$ be the ring of Laurent polynomials in $\hq$.
Let $\mathcal{H}$ be a free $\bZ [q^{\pm \frac{1}{2}}]$-module with formal basis $\{\nT_w\}_{w \in W}$.
The multiplication on $\mathcal{H}$ is defined by
\begin{equation*}
  \nT_s \nT_w := \begin{cases}
                  \nT_{sw} , & \mbox{if } \ell(sw) > \ell(w) \\
                  (\hq - \mhq) \nT_w + \nT_{sw} , & \mbox{if } \ell(sw) < \ell(w)
                 \end{cases}, \quad \forall s \in S, w \in W.
\end{equation*}
Then, $\mathcal{H}$ is an associative $\bZ [q^{\pm \frac{1}{2}}]$-algebra with unity $\nT_e$.
Here $\nT_w = \mhq[\ell(w)] T_w$ where $T_w$ is the standard basis element in \cite{KL79}.

Besides the normalized standard basis $\{\nT_w \mid w \in W\}$, we have the basis $\{C_w \mid w \in W\}$ in \cite{KL79}, such that
\begin{align}
    C_w & = \sum_{y \in W} (-1)^{\ell(w)+\ell(y)} q^{\frac{1}{2} (\ell(w) - \ell(y))} P_{y,w}(q^{-1}) \nT_y   \label{eq-Cw} \\
     & = \sum_{y \in W} (-1)^{\ell(w)+\ell(y)} q^{-\frac{1}{2} (\ell(w) - \ell(y))} P_{y,w}(q) \nT_{y^{-1}}^{-1}, \notag %
\end{align}
where $P_{y,w} \in \mathbb{Z}[q]$, $P_{y,w} = 0$ unless $y \leq w$, $P_{w,w} = 1$, and $\deg_q P_{y,w} \leq \frac{1}{2} (\ell(w) - \ell(y) - 1)$ if $y < w$.
Here $\le$ is the Bruhat order on $W$, and $y < w$ indicates $y \le w$ and $y \ne w$.



\begin{example} \leavevmode \label{eg-C}
 \begin{enumerate}
   \item \label{eg-C-1} $C_e = \nT_e$; $C_s = \nT_s - \hq$ for $s \in S$.
   \item \label{eg-C-2} If $r, t \in S$ with $r \ne t$ and  $m_{rt} = m < \infty$, we write for simplicity $r_{k} := rtr \cdots$ (product of $k$ factors), $t_k = trt \cdots$ ($k$ factors), and $w_{rt} := r_{m} = t_{m}$. Then, $P_{y, w_{rt}} = 1$ for any $y = r_k$ or $t_k$ ($0 \le k \le m$). Thus,
       \begin{equation} \label{Cwrt}
         C_{w_{rt}}  = \nT_{w_{rt}} - \hq (\nT_{r_{m-1}} + \nT_{t_{m-1}}) + \cdots  
               + (-1)^i \hq[i] (\nT_{r_{m-i}} + \nT_{t_{m-i}}) + \dots + (-1)^m \hq[m].
       \end{equation}
 \end{enumerate}
\end{example}

\begin{notation}
  We keep the notations $r_k, t_k$ and $w_{rt}$ in Example \ref{eg-C}\eqref{eg-C-2} throughout the paper.
\end{notation}

For any $x, y, w \in W$, let $h_{x, y, w} \in \bZ [q^{\pm \frac{1}{2}}]$ be such that $C_x C_y = \sum_{w \in W} h_{x,y,w} C_w$.
Following Lusztig \cite{Lusztig85-cell-i}, the function $\af$ on $W$ is defined by $\af(w) := \min \{i \in \mathbb{N} \mid \hq[i] h_{x,y,w} \in \mathbb{Z}[\hq], \forall x, y \in W\}$.
It turns out that $\af(w)$ is a well-defined natural number (see, for example, \cite[Propsition 1.2]{Lusztig87-cell-ii}).
We say the function $\af$ is \emph{bounded} if there is $N \in \mathbb{N}$ such that $\af(w) \le N$ for any $w \in W$.

  It is a long-standing conjecture that Lusztig's function $\af$ is bounded on any Coxeter group $(W,S)$ of finite rank (that is, $\abs{S} < \infty$).
  Moreover, it is also conjectured that the function $\af$ is bounded by $N(W)$, where $N(W)$ is the maximal length of longest elements of all finite parabolic subgroups of $W$.
  See \cite[Question 1.13(iv)]{Xi94} and \cite[Conjecture 13.4]{Lusztig14-hecke-unequal}.
  The bound was verified in some special cases (see for example \cite{Belolipetsky04, LS19, Lusztig85-cell-i, Xi12, Zhou13}).
  Recently, it is proved in \cite{Chen25} that the function $\af$ admits an upper bound which is possibly greater then $N(W)$:

\begin{theorem}
  Let $(W,S)$ be a Coxeter group of finite rank. Then, there exists a constant $N'(W)$ such that the function $\af$ is bounded by $N'(W)$.
\end{theorem}

For $x, y \in W$, we say $y \lrle x$ if there exist $ H_1, H_2 \in \mathcal{H}$, such that $C_y$ has nonzero coefficient in the expression of $H_1 C_x H_2$ with respect to the basis $\{C_w\}_{w \in W}$.
We say $x \lrsim y$ if $x \lrle y$ and $y \lrle x$.
It turns out that $\lrle$ is a pre-order on $W$, and $\lrsim$ is an equivalence relation on $W$. The equivalence classes are called \emph{two-sided cells}.
The set of two-sided cells forms a partial order set with respect to $\lrle$.

We have the following standard properties of Kazhdan-Lusztig basis $\{C_w\}_{w \in W}$, two-sided cells, and Lusztig's function $\af$.

\begin{proposition} [\cite{Lusztig85-cell-i, Lusztig87-cell-ii, Lusztig14-hecke-unequal}] \label{prop-cell-a}
  Let $s \in S$ and $w, x, y \in W$.
  \begin{enumerate}
    \item \label{eg-C-3} $P_{y,w}(0) = 1$ for any $y \leq w$.
    \item \label{prop-cell-a-1} Suppose $sw > w$.
        Then, $C_sC_w  = C_{sw} + \sum \mu_{y,w} C_y$ where the sum runs over all $y < w$ such that $\deg_q P_{y,w} = \frac{1}{2} (\ell(w) - \ell(y) - 1)$ and $s y < y$, and $\mu_{y,w}$ is the coefficient of the top-degree term in $P_{y,w}$.
        The formula is similar for the case $ws > w$.
        In particular, $sw \lrle w$ if $sw > w$, $ws \lrle w$ if $ws > w$.
    \item \label{prop-cell-a-2} $\af(w) = 0$ if and only if $w = e$.
        We have $\af(s) = 1$ for any $s \in S$.
    \item \label{prop-cell-a-3} If $m_{rt} < \infty$ for some $r, t \in S$ with $r \ne t$, then $\af(w_{rt}) = m_{rt}$. 
    \item \label{prop-cell-a-4} If $y \lrle x$, then $\af(y) \geq \af(x)$.
        If $y \lrsim x$, then $\af(y) = \af(x)$.
    \item \label{prop-cell-a-5} (Given that $\af$ is bounded on $W$.) If $\af(y) = \af(x)$ and $y \lrle x$, then $y \lrsim x$.
  \end{enumerate}
\end{proposition}

\begin{notation}
   By Proposition \ref{prop-cell-a}\eqref{prop-cell-a-4}, we may write $\af(\mathcal{C})$ for a two-sided cell $\mathcal{C}$ to indicate the number $\af(x)$ where $x$ is an arbitrary element in $\cC$.
   We also write $\mathcal{C} \lrle w$ to indicate that $x \lrle w$ for some $x \in \cC$.
   Similar for $w \lrle \cC$ and $\cC^\prime \lrle \cC$.
\end{notation}

We shall need the following result due to H. Matsumoto and J. Tits.
See also \cite[Theorem 1.9]{Lusztig14-hecke-unequal}.

\begin{theorem} [\cite{Matsumoto64, Tits69}] \label{lemma-two-expr}
  Let $s_1 \cdots s_n$ and $s_1^\prime \cdots s_n^\prime$ be two reduced expressions of some element $w \in W$.
  Then, we can obtain one of the expressions from the other by finite steps of replacement of the form
    \begin{equation*}
          \underbrace{rtr\cdots} _{m_{rt} \text{ factors}} = \underbrace{trt\cdots} _{m_{rt} \text{ factors}} \quad (r,t \in S, r \ne t, m_{rt} < \infty).
    \end{equation*}
\end{theorem}

If the Coxeter group $(W,S)$ is irreducible (that is, the Coxeter graph is connected), there is a two-sided cell with a simple description, formulated by Lusztig \cite{Lusztig83-intrep} as follows.

\begin{lemma} [See also {\cite[Proposition 3.8]{Lusztig83-intrep}}] \leavevmode \label{lem-C1}
  \begin{enumerate}
    \item \label{lem-C1-1} For $w \in W \setminus \{e\}$, $\af(w) = 1$ if and only if $w$ has a unique reduced expression.
    \item \label{lem-C1-2} Let $\mathcal{C}_1 := \{w \in W \mid \af(w) = 1\}$. If $(W,S)$ is irreducible, then $\mathcal{C}_1$ is a two-sided cell. Moreover, for any $x \in W \setminus \{e\}$, we have $x \lrle \mathcal{C}_1$.
  \end{enumerate}
\end{lemma}

\begin{proof}
  If $w$ has two different reduced expressions, then by Theorem~\ref{lemma-two-expr} $w$ has a reduced expression of the form
  \begin{equation*}
    w = \cdots \underbrace{(rtr\cdots)} _{m_{rt} \text{ factors}} \cdots \quad (r,t \in S, r \ne t, m_{rt} < \infty).
  \end{equation*}
  Then, $w \lrle w_{rt}$ by Proposition \ref{prop-cell-a}\eqref{prop-cell-a-1}.
  Thus, $\af(w) \ge \af(w_{rt}) = m_{rt} \ge 2$ by Proposition \ref{prop-cell-a}\eqref{prop-cell-a-3}\eqref{prop-cell-a-4}.
  Conversely, suppose that the reduced expression of $w$ is unique and $sw < w$, $s\in S$.
  We have $w \lrsim s$ by \cite[Proposition 3.8]{Lusztig83-intrep}.
  Therefore, $\af(w) = 1$ by Proposition \ref{prop-cell-a}\eqref{prop-cell-a-2}\eqref{prop-cell-a-4}.
  The point \eqref{lem-C1-1} of this lemma is proved.

  The first assertion in the point \eqref{lem-C1-2} has been proved in \cite[Proposition 3.8]{Lusztig83-intrep}.
  The second one is deduced form Proposition \ref{prop-cell-a}\eqref{prop-cell-a-1} and the fact that $s \in \cC_1$.
\end{proof}

\begin{remark}
  A proof of Lemma \ref{lem-C1}\eqref{lem-C1-1} is also provided in \cite[Corollary 3.1]{Xu19} which uses the boundedness of the function $\af$ (the proof there uses the result stated in Proposition \ref{prop-cell-a}\eqref{prop-cell-a-5}).
\end{remark}

In general, if we decompose the Coxeter graph into connected components, say $S = \sqcup_i S_i$, then $\mathcal{C}_1 = \sqcup_i \mathcal{C}_{1,i}$ is a union of two-sided cells, where $\mathcal{C}_{1,i} = \{w \in W \mid \af(w) = 1, \text{ and $w$ is a product of elements in $S_i$}\}$.

The following result (by B.\ Elias and G.\ Williamson \cite{EW14}) is a milestone in this subject and its proof is highly nontrivial.

\begin{theorem} [{\cite[Corollary 1.2]{EW14}}] \label{lemma-pos}
  For any $x,y,w \in W$, it holds $P_{y,w} \in \mathbb{N}[q]$.
  If we write $h_{x,y,w} = \sum_{i \in \mathbb{Z}} c_i \hq[i]$ where each $c_i \in \bZ$, then $(-1)^i c_i \in \mathbb{N}$.
\end{theorem}

\subsection{The cell representations}

We will consider complex representations of Coxeter groups.
For convenience, we define $\cH_\bC := \cH \otimes_\bZ \bC$ to be the Hecke algebra with coefficient ring $\AC$.
By convention, a \emph{representation} (over $\bC$) of $\cH_\bC$ is defined to be a $\mathbb{C}$-vector space endowed with an $\mathcal{H}_\bC$-module structure, where $\hq$ acts by a $\mathbb{C}^\times$-scalar multiplication.

Let $\chi: \AC \to \bC$ be a homomorphism of $\bC$-algebras (called a \emph{specialization}) and $\mathcal{C}$ be a two-sided cell.
We define $\mathcal{J}_\mathcal{C}$ to be a $\bC$-vector space with formal basis $\{J_w \mid w \in \mathcal{C}\}$.
The following formula defines an $\mathcal{H}_\bC$-module structure on $\mathcal{J}_\mathcal{C}$,
\begin{equation*}
  C_x \cdot J_y := \sum_{w \in \mathcal{C}} \chi(h_{x,y,w}) J_w, \quad \forall x \in W, y \in \mathcal{C}.
\end{equation*}
This representation of $\cH_\bC$ is called the \emph{cell representation provided by} $\cC$. We denote it also by $\cJ_\cC^\chi$ if the specialization $\chi$ needs to be emphasized.

\begin{remark} \label{rmk-cell-rep}
  The cell representation $\cJ_\cC$ has the following property.
  If $C_x \cdot \cJ_\cC \ne 0$, then there exists $y \in \cC$ such that $h_{x, y, w} \ne 0$ for some $w \in \cC$, and consequently $\cC \lrle x$.
  In particular, $\af(x) \le \af(\cC)$.
\end{remark}

The following proposition is due to Lusztig.
But a sketched proof is presented here because the author did not find an explicit statement in the literature.

\begin{proposition} [Lusztig, 1980's] \label{prop-cell-rep-m} Suppose $V$ is a nonzero irreducible representation of $\mathcal{H}_\bC$ on which $\hq$ acts by a scalar $\chi(\hq)$.
  \begin{enumerate}
    \item \label{prop-cell-rep-m1} Suppose there exists a two-sided cell $\cC$ such that
        \begin{enumerate}
          \item \label{prop-cell-rep-i} for any $x \in W$, if $C_x \cdot V \neq 0$, then we have $\mathcal{C} \lrle x$,
          \item \label{prop-cell-rep-ii} there is an element $w \in \mathcal{C}$ such that $C_w \cdot V \neq 0$.
        \end{enumerate}
        Then, such a two-sided cell is unique.
        Moreover, $V$ is a simple quotient of $\mathcal{J}_\mathcal{C}^\chi$.
    \item \label{prop-cell-rep-m3} (Given that the function $\af$ is bounded on $W$.)
        There exists a two-sided cell $\cC$ satisfying the conditions \eqref{prop-cell-rep-i}\eqref{prop-cell-rep-ii} in \eqref{prop-cell-rep-m1}.
  \end{enumerate}
\end{proposition}

\begin{proof}
  For \eqref{prop-cell-rep-m1}, suppose there exists another two-sided cell $\cC^\prime$ satisfying the conditions \eqref{prop-cell-rep-i}\eqref{prop-cell-rep-ii}.
  Then, by the condition \eqref{prop-cell-rep-i} for $\cC$ and the condition \eqref{prop-cell-rep-ii} for $\cC^\prime$, we deduce that $\cC \lrle \cC^\prime$, and vice versa.
  Thus, $\cC = \cC^\prime$ and the two-sided cell $\cC$ is unique.
  Next we take $v \in V$ such that $C_w \cdot v \neq 0$ for some $w \in \mathcal{C}$.
  Then, it can be verified that the map $\varphi: \mathcal{J}_\mathcal{C}^\chi \to V$, $J_y \mapsto C_y \cdot v$ is a well-defined surjective homomorphism of $\mathcal{H}_\bC$-modules.
  Thus, $V \simeq \mathcal{J}_\mathcal{C}^\chi / \ker \varphi$ is a quotient of $\cJ_\cC^\chi$.

  For \eqref{prop-cell-rep-m3}, let $\mathcal{C}$ be a two-sided cell satisfying the condition \eqref{prop-cell-rep-ii} such that $\af(\cC)$ is maximal (such a two-sided cell exists since $\af$ is  bounded and $C_e = \nT_e$ acts by identity).
  Then, by Proposition \ref{prop-cell-a}\eqref{prop-cell-a-5}, we have
  \begin{equation*}
    \text{for any $y \in W$, if $y \lrle \mathcal{C}$ and $y \notin \mathcal{C}$, then $C_y \cdot V = 0$.}
  \end{equation*}
  Then, $V$ is a quotient of $\cJ_\cC^\chi$ by the same arguments in the proof of  \eqref{prop-cell-rep-m1}.
  Suppose now $x \in W$ and $C_x \cdot V \neq 0$.
  Then, $C_x \cdot \cJ_\cC^\chi \neq 0$ and hence there exists $y \in \mathcal{C}$ such that $C_x \cdot J_y \neq 0$.
  In other words, there exist $y, z \in \mathcal{C}$ such that $h_{x,y,z} \neq 0$.
  Thus, $z \lrle x$ and the cell $\mathcal{C}$ satisfies the statement \eqref{prop-cell-rep-i}.
\end{proof}

Immediately we have the following corollary on the lowest two-sided cell, answering the question \cite[Question 1.13(ii)]{Xi94}.

\begin{corollary} \label{cor-cell-low}
  Suppose $N = \max \{\af(w) \mid w \in W\} <\infty$.
  Then, $\mathcal{C}_{\text{\upshape low}} := \{w \in W \mid \af(w) = N\}$ is a two-sided cell. Moreover, for any $w \in W$ we have $\mathcal{C}_{\text{\upshape low}} \lrle w$.
\end{corollary}

\begin{proof}
  Let $\varepsilon: s \mapsto -1, \forall s \in S$ be the sign representation of $W$.
  We regard $\varepsilon$ as a representation of $\cH_\bC$ on which $\hq$ acts by 1.
  Then, by the formula \eqref{eq-Cw}, we see that for any $w \in W$ the element $C_w$ acts by the scalar $\varepsilon(C_w) = (-1)^{\ell(w)} \sum_{y} P_{y,w}(1)$.
  By Proposition \ref{prop-cell-a}\eqref{eg-C-3} and Theorem~\ref{lemma-pos}, this number is nonzero.
  Let $\mathcal{C}$ be the two-sided cell attached to $\varepsilon$ as in Proposition \ref{prop-cell-rep-m}.
  Then, Proposition \ref{prop-cell-rep-m} implies that $\mathcal{C} \lrle w$ for any $w$.
  By Proposition \ref{prop-cell-a}\eqref{prop-cell-a-4}, we have $\af(\cC) \ge \af(w)$ for any $w \in W$.
  Now we choose $w \in W$ such that $\af(w) = N$.
  Then, $\af(\cC) = N$.
  By definition, we have $\cC \subseteq \cC_\text{low}$.

  Conversely, for any $w \in \cC_\text{low}$, that is, $\af(w) = N$, we have seen that $\cC \lrle w$.
  Thus, $w \in \cC$ by Proposition \ref{prop-cell-a}\eqref{prop-cell-a-5}.
  Therefore, $\cC = \cC_\text{low}$ and $\cC_\text{low}$ is a two-sided cell.
\end{proof}

\begin{remark}
  The lowest two-sided cell was first studied for affine Weyl groups in \cite{Lusztig85-cell-i, Shi87}.
  It was later studied for general Coxeter groups in \cite[Theorem 1.5]{Xi12} with the assumption that $N(W)$ is a bound for the function $\af$.
  See also \cite[Theorem~2.1]{Xie17}.
\end{remark}


\begin{remark} \label{rmk-sign-1}
  From the proof of Corollary \ref{cor-cell-low}, we see that the statement that ``the function $\af$ is bounded'' is equivalent to saying that there exists a two-sided cell attached to the sign representation $\varepsilon$ in the sense of Proposition \ref{prop-cell-rep-m}\eqref{prop-cell-rep-m1}.
\end{remark}

\section{Coxeter group representations of \texorpdfstring{$\af$}{a}-function value 1} \label{sec-main-1}
\subsection{Overview}
Let $(W,S)$, $\cH_\bC$, $C_w$, etc be as in the last section.
The representations we consider are all representations over the field $\bC$.

Following \cite[Proof of Lemma 1.9]{Lusztig87-cell-iii}, we define a function $\af$ on the set of representations of $\mathcal{H}_\mathbb{C}$ as follows.

\begin{definition} \label{def-a}
Suppose $V$ is a nonzero representation of $\cH_\bC$, not necessary to be irreducible.
If there exists a natural number $N$ such that
\begin{enumerate}
  \item for each $x \in W$, if $\af(x) > N$, then we have $C_x \cdot V = 0$, and
  \item there exists an element $w \in W$ such that $\af(w) = N$ and $C_w \cdot V \neq 0$,
\end{enumerate}
then we say $N$ is the \emph{$\af$-function value} of $V$.
\end{definition}

In particular, if $V$ is irreducible and there exists a two-sided cell $\cC$ attached to $V$ in the sense of Proposition \ref{prop-cell-rep-m}\eqref{prop-cell-rep-m1}, then $\af(\cC)$ is the $\af$-function value of $V$.

From now on, we shall adopt the following convention.

\begin{convention} \label{conv-chiq=1} \leavevmode
  \begin{enumerate}
    \item In the rest of this paper, we choose the specialization $\chi: \AC \to \bC$ so that $\chi(\hq) = 1$.
        Then, a representation of $\cH_\bC$ on which $\hq$ acts by $\chi(\hq) = 1$ is the same thing as a representation of $W$.
        In this way, the $\af$-function value of a representation of $W$ is defined (c.f., Definition \ref{def-a}).
    \item For an arbitrary element $h = \sum_{w \in W} a_w \nT_w \in \cH_\bC$ ($a_w \in \AC$), we denote again by $h$ the element $\sum_{w \in W} \chi(a_w) w$ in the group algebra $\bC[W]$.
        In particular, in $\bC[W]$ we denote by $C_w$ the element $\sum_{y \in W} (-1)^{\ell(w)+\ell(y)} P_{y,w}(1) y$.
  \end{enumerate}
\end{convention}


\begin{example} \label{eg-a=0}
  Recall that $C_e = \nT_e$, $C_s = \nT_s - \hq \in \cH_\bC$ (see Example \ref{eg-C}\eqref{eg-C-1}).
  They are specialized to $C_e = e$, $C_s = s - e \in \bC[W]$.
  Therefore, a nonzero representation $V$ of $W$ is of $\af$-function value $0$ if and only if the action of $W$ on $V$ is trivial (note that $\af(w) = 0$ if and only if $w = e$, see Proposition \ref{prop-cell-a}\eqref{prop-cell-a-2}).
\end{example}

The main result of this section is the following characterization of representations of $\af$-function value $1$.

\begin{theorem} \label{thm-A1}
Let $V$ be a nonzero representation of $W$.
Then, the following two conditions are equivalent:
\begin{enumerate}
  \item \label{thm-A1-1} for any $w \in W$ with $\af(w) > 1$, we have $C_w \cdot V = 0$;
  \item \label{thm-A1-2} there is no $v \in V \setminus \{0\}$, such that $r \cdot v = t \cdot v = -v$ for some $r, t \in S$ with $r \ne t$ and $m_{rt} < \infty$.
\end{enumerate}
\end{theorem}

\begin{remark} \label{rmk-a=1} 
 \begin{enumerate}
   \item This theorem is saying that a nonzero nontrivial representation $V$ of $W$ is of $\af$-function value $1$ if and only if for any $r,t \in S$ with $r \ne t$ and $m_{rt} < \infty$ there is no common eigenvector of $r,t$ with eigenvalue $-1$.
       In particular, the geometric representation of $W$ (see \cite[Ch.\ V, \S 4]{Bourbaki-Lie456}) is of $\af$-function value $1$.
       Some generalizations of this representation (see for example \cite{BB05, BCNS15, Donnelly11, Hee91, Krammer09, Vinberg71}) are also of $\af$-function value $1$.
   \item \label{rmk-a=1-2} In another paper \cite{Hu23}, the author defined and classified two classes of representations, which are called generalized geometric representations and R-representations.
       All of these representations are of $\af$-function value $1$.
       In Section \ref{sec-main-2}, we will see that for certain simply laced Coxeter groups all the irreducible representations of $\af$-function value $1$ are those so-called R-representations.
   \item \label{rmk-a=1-3} In paper \cite{Hu22}, the author constructs some irreducible representations of infinite dimension for certain Coxeter groups.
       These representations are of $\af$-function value $1$ as well. 
       This indicates that the representations of $\af$-function value $1$ might be complicated in general.
       See also  Remark \ref{rmk-simply-laced}\eqref{rmk-simply-laced-2}.
   \item Recall Lemma \ref{lem-C1} that we have the two-sided cell $\cC_1$ when $(W,S)$ is irreducible.
       By Proposition \ref{prop-cell-rep-m}\eqref{prop-cell-rep-m1}, an irreducible representation of $\af$-function value $1$ is a quotient of the cell representation $\cJ_{\cC_1}^\chi$.
   \item \label{rmk-a=1-5} Suppose $I \subseteq S$ is a subset and the parabolic subgroup $W_I := \langle I \rangle$ generated by $I$ is finite.
       Suppose moreover there exists a nonzero vector $v$ in $V$ such that $s \cdot v = - v$ for any $s \in I$.
       Then, it can be shown that the $\af$-function value of $V$ is not less than $\ell(w_I)$, where $w_I$ is the longest element in $W_I$.
       In view of this fact and Theorem \ref{thm-A1}, as well as the proof of Corollary \ref{cor-cell-low}, we feel that the $\af$-function value of a representation of $W$ might be related to the pattern in which elements in $S$ share common eigenvectors with eigenvalue $-1$.
 \end{enumerate}
\end{remark}

\subsection{Proof of Theorem \ref{thm-A1}}

To prove Theorem \ref{thm-A1}, we need the following lemma.

\begin{lemma} \label{lem-Cwrt}
  Let $D_m := \langle r,t \mid r^2 = t^2 = (rt)^m = e \rangle$ be a finite dihedral group ($m \in \bN_{\ge 2}$).
  Let $\rho_k, \II, \varepsilon$ and $\varepsilon_r, \varepsilon_t$ (if $m$ is even) denote the representations defined in Subsection \ref{subsec-dih}.
  We also regarded these representations as homomorphisms from the group algebra $\bC[D_m]$ to $GL_1(\bC)$ or $GL_2(\bC)$.
  \begin{enumerate}
    \item \label{lem-Cwrt-1} If $1 \le k < \frac{m}{2}$, then $\rho_k(C_{w_{rt}}) = 0$.
    \item \label{lem-Cwrt-2} $\II(C_{w_{rt}}) = 0$, $\varepsilon (C_{w_{rt}}) \ne 0$.
    \item \label{lem-Cwrt-3} If $m$ is even, then $\varepsilon_r (C_{w_{rt}}) = \varepsilon_t (C_{w_{rt}}) = 0$.
  \end{enumerate}
\end{lemma}

\begin{proof}
  Recall \eqref{Cwrt} that $C_{w_{rt}}  = \nT_{w_{rt}} + (-1)^m \hq[m] + \sum_{i = 1}^{m-1} (-1)^i \hq[i] (\nT_{r_{m-i}} + \nT_{t_{m-i}})$.
  After our specialization, we have
  \begin{align*}
    \bC[D_m] \ni C_{w_{rt}} & = w_{rt} + (-1)^m + \sum_{i = 1}^{m-1} (-1)^i (r_{m-i} + t_{m-i}) \\
    & = (-1)^m \bigl[ (e + r_2 + t_2 + r_4 + t_4 + \cdots + w_{rt}) \\
     & \mathrel{\phantom{= (-1)^m \bigl[}} - (r + t + r_3 + t_3 + \cdots + r_{m-1} + t_{m-1}) \bigr] \\
    & \mathrel {\phantom{=}} \text{(if $m$ is even)} \\
     \mathrel{\text{or } } & =  (-1)^m \bigl[ (e + r_2 + t_2 + r_4 + t_4 + \cdots + r_{m-1} + t_{m-1}) \\
     & \mathrel{\phantom{= (-1)^m \bigl[}} - (r + t + r_3 + t_3 + \cdots + w_{rt} ) \bigr] \\
     & \mathrel {\phantom{=}} \text{(if $m$ is odd).}
  \end{align*}

  Recall Subsection \ref{subsec-dih} that for any integer $k$ satisfying $1 \le k \le \frac{m}{2}$, $\rho_k$ is a representation of dimension 2 with a real form.
  On this real plane, $r_i$ is a rotation (of angle $\frac{ik\uppi}{m}$) if $i$ is even (we regard $e$ as the rotation of angle 0), and is a reflection if $i$ is odd.
  Similar for $t_i$.
  Therefore, $\rho_k(C_{w_{rt}})$ is the difference between the $m$ rotations and the $m$ reflections up to a sign which depends on the parity of $m$.
  But notice that the sum of the $m$ rotations is zero, and so is the sum of the $m$ reflections.
  Thus, $\rho_k(C_{w_{rt}}) = 0$.

  In particular, if $m$ is even and $k = \frac{m}{2}$, then $\rho_\frac{m}{2} \simeq \varepsilon_r \oplus \varepsilon_t$.
  Hence, $\varepsilon_r (C_{w_{rt}}) = \varepsilon_t (C_{w_{rt}}) = 0$.

  The equality $\II(C_{w_{rt}}) = 0$ is straightforward.
  At last, note that $\varepsilon(r_i) = \varepsilon(t_i) = (-1)^i$.
  Thus, $\varepsilon(C_{w_{rt}}) = (-1)^m (m - (-m)) \ne 0$.
\end{proof}

\begin{lemma} \label{lem-Cwrt=0}
  Suppose $V$ is a nonzero representation of $W$ satisfying the condition \eqref{thm-A1-2} of Theorem \ref{thm-A1}.
  Suppose $r,t \in S$ such that $r \ne t$ and $m_{rt} < \infty$.
  Then, $C_{w_{rt}} \cdot V = 0$.
\end{lemma}

\begin{proof}
  The subgroup $\langle r,t \rangle$ of $W$ generated by $r$ and $t$ is a finite dihedral group $D_{m_{rt}}$.
  We consider the restriction of the representation $V$ to the subgroup $\langle r,t \rangle$.
  Then, $V$ (either finite or infinite dimensional) is decomposed into a direct sum of irreducible representations of $D_{m_{rt}}$ since the group algebra $\bC[D_{m_{rt}}]$ is semisimple.
  By the assumption, the sign representation $\varepsilon$ does not appear in this direct sum.
  By Lemma \ref{lem-Dm-rep} and Lemma \ref{lem-Cwrt}, we see that $C_{w_{rt}}$ acts by zero on every irreducible representation of $D_{m_{rt}}$ except $\varepsilon$.
  Thus, $C_{w_{rt}} \cdot V = 0$.
\end{proof}

\begin{proof}[Proof of Theorem \ref{thm-A1}]
  Suppose there exist $r,t \in S$ with $r \ne t$ and $m_{rt} < \infty$, and a nonzero vector $v \in V$ such that $r \cdot v = t\cdot v = -v$.
  Then, $v$ spans a subrepresentation in $V$ for the dihedral subgroup $\langle r,t \rangle \simeq D_{m_{rt}}$.
  This one-dimensional subrepresentation is isomorphic to the sign representation $\varepsilon$ of $\langle r,t \rangle$.
  By Lemma \ref{lem-Cwrt}\eqref{lem-Cwrt-2}, we have $C_{w_{rt}} \cdot v \ne 0$.
  On the other hand, by Proposition \ref{prop-cell-a}\eqref{prop-cell-a-3} we have $\af(w_{rt}) = m_{rt} \ge 2$.
  Thus, the statement \eqref{thm-A1-1} implies the statement \eqref{thm-A1-2}.

  Conversely, suppose $V$ satisfies the condition \eqref{thm-A1-2}.
  Suppose $w \in W$ and $\af(w) > 1$.
  We do induction on $\ell(w)$ to show that $C_w \cdot V = 0$.
  First of all, we must have $\ell(w)  \ge 2$.
  Otherwise if $\ell(w) = 0$ or $1$, then $\af(w) = \ell(w)$ by  Proposition \ref{prop-cell-a}\eqref{prop-cell-a-2}.
  If $\ell(w) = 2$, then by Lemma \ref{lem-C1}\eqref{lem-C1-1} we have $w = w_{rt}$ for some $r,t \in S$ with $m_{rt} = 2$.
  In this case $C_w \cdot V = 0$ by Lemma \ref{lem-Cwrt=0}.

  Suppose now $\ell(w) > 2$.
  By Lemma \ref{lem-C1}\eqref{lem-C1-1}, Theorem~\ref{lemma-two-expr} and the assumption that $\af(w) > 1$, $w$ has a reduced expression of the form
  \begin{equation*}
    w = \cdots \underbrace{(rtr\cdots)} _{m_{rt} \text{ factors}} \cdots \quad (r,t \in S, r \ne t, m_{rt} < \infty).
  \end{equation*}
  That is, we can write $w = x w_{rt} y$ where $x, y \in W$ and $\ell(w) = \ell(x) + m_{rt} + \ell(y)$.
  If $x = y = e$, then $w = w_{rt}$ and $C_w \cdot V = 0$ by Lemma \ref{lem-Cwrt=0}.
  If $x \ne e$, then there exists $s \in S$ such that $\ell(s x) < \ell(x)$.
  It follows that $\ell(sw) < \ell(w)$.
  Then, by Proposition \ref{prop-cell-a}\eqref{prop-cell-a-1} we have
  \begin{equation*}
    C_s C_{sw} = C_w + \sum_{\substack{ z < sw, sz < z \\ \deg_q P_{z,sw} = \frac{1}{2} (\ell(sw) - \ell(z) - 1)}} \mu_{z,sw} C_z.
  \end{equation*}
  In particular, those $z$ which occur in the summation satisfy $\ell(z) < \ell(sw)$, $z\lrle sw$, and then $\af(z) \ge \af(sw)$ by Proposition \ref{prop-cell-a}\eqref{prop-cell-a-4}.
  Notice that $sw$ also has a reduced expression containing a consecutive product $rtr \cdots$ ($m_{rt}$ factors) since $sw = (sx)w_{rt}y$ and $\ell(sw) = \ell(sx) + m_{rt} + \ell(y)$.
  Then, we have $\af(sw) > 1$ by Lemma \ref{lem-C1}\eqref{lem-C1-1}.
  By induction hypothesis, we have  $C_{sw} \cdot V = 0$ and $C_z \cdot V = 0$ for those $z$ which occur in the summation.
  Therefore, $C_w \cdot V = (C_s C_{sw} - \sum_{z} \mu_{z,sw} C_z) \cdot V = 0$.
  The case where $y \ne e$ is similar: if $\ell(ys) < \ell(y)$ for some $s \in S$, then $C_w = C_{ws} C_s - \sum_z \mu_{z, ws} C_z$ and $C_{ws} \cdot V = C_z \cdot V = 0$ for those $z$ which occur.
\end{proof}

\section{Simply laced Coxeter groups} \label{sec-main-2}
\subsection{Overview}
A Coxeter group $(W,S)$ is called \emph{simply laced} if $m_{st} = 2$ or $3$ for any $s,t \in S$.
The main result (Theorem \ref{thm-A1=R}) of this section states that for certain simply laced Coxeter groups we can determine all the irreducible representations of $W$ of $\af$-function value $1$.
In other words, for such Coxeter groups we are able to determine all the simple quotients of the cell representation $\cJ_{\cC_1}^\chi$, where $\chi(\hq) = 1$ as we set in Convention \ref{conv-chiq=1}.
It turns out that these irreducible representations of $\af$-function value $1$ are the so-called R-representations defined and classified in \cite{Hu23}.
Before presenting the precise statement, we shall introduce some related notations and terminologies.

\begin{notation}
  For a representation $(V, \rho)$ of $W$ and for $s \in S$, we denote by $V_s^+$ and $V_s^-$ the eigen-subspaces of $\rho(s)$ in $V$ with eigenvalues $+1$ and $-1$ respectively.
\end{notation}

Clearly, we have $V = V_s^+ \bigoplus V_s^-$ as a vector space for any $s \in S$.

The R-representations defined in \cite{Hu23} are a class of representations on which each $s \in S$ acts by an abstract reflection.
For a simply laced Coxeter group $(W,S)$, the definition of R-representations  reduces to the following.

\begin{definition} \label{def-R}
  A representation $(V,\rho)$ of a simply laced Coxeter group $(W,S)$ is called an \emph{R-representation} if
  \begin{enumerate}
    \item \label{def-R-1} $\sum_{s \in S} V_s^- = V$,
    \item \label{def-R-2} for any $r,t \in S$ with $r \ne t$, we have $V_r^- \cap V_t^- = 0$,
    \item \label{def-R-3} for any $s \in S$, we have $\dim V_s^- = 1$.
  \end{enumerate}
\end{definition}

Clearly, an R-representation is of $\af$-function value $1$ by Theorem \ref{thm-A1}.
Now we can state the main theorem of this section.
Recall that a Coxeter group $(W,S)$ is said to be irreducible if its Coxeter graph $G$ is connected.

\begin{theorem} \label{thm-A1=R}
  Let $(W,S)$ be an irreducible simply laced Coxeter group.
  Suppose there is at most one cycle in the Coxeter graph $G$.
  Then, any irreducible representation of  $W$ of $\af$-function value $1$ is an R-representation.
  In particular, such representations are finite dimensional (with dimension not greater than $\abs{S}$).
\end{theorem}

\begin{remark} \label{rmk-simply-laced} \leavevmode
  \begin{enumerate}
    \item \label{rmk-simply-laced-1} The irreducible R-representations are classified in \cite{Hu23}.
        In particular, if $G$ is a tree, then such representation is unique (in fact, it is the simple quotient of the geometric representation defined in \cite[Ch.\ V, \S 4]{Bourbaki-Lie456}).
        On the other hand,
        if $G$ has a cycle, then such representations are parameterized by $\bC^\times$.
        For the reader's convenience, we present the construction of such representations in these settings in Subsection \ref{subsec-R}.
    \item \label{rmk-simply-laced-2} In contrast, if there are at least two cycles in a Coxeter graph, or if there is at least one cycle in a non-simply-laced Coxeter graph, then we can 
        construct an infinite dimensional irreducible representation, which is of $\af$-function value $1$, for the corresponding Coxeter group.
        The construction is motivated by the proof of Theorem \ref{thm-A1=R}.
        See \cite{Hu22} for details.
  \end{enumerate}
\end{remark}

\subsection{Proof of Theorem \ref{thm-A1=R}}

This subsection is devoted to proving Theorem \ref{thm-A1=R}.

Let $(W,S)$ be an irreducible simply laced Coxeter group, and $(V,\rho)$ be an irreducible representation of $W$ of $\af$-function value $1$.
We first collect some elementary results which we will use later.

\subsubsection{Basic results}
\begin{lemma} \label{lem-R} \leavevmode
  \begin{enumerate}
    \item \label{lem-R-1} For any $v \in V$ and $s \in S$, we have $s \cdot v - v \in V_s^-$.
    \item \label{lem-R-2} We have $\sum_{s \in S} V_s^- = V$.
    \item \label{lem-R-3} For any distinct elements $r, t \in S$, we have $V_r^- \cap V_t^- = 0$.
    \item \label{lem-R-4} For any $r, t \in S$ with $m_{rt} = 2$, and for any $v \in V_t^-$, we have $r \cdot v = v$.
    \item \label{lem-R-5} For any distinct elements $r, t \in S$, we have $\dim V_r^- = \dim V_t^-$ (allowed to be infinity).
    \item \label{lem-R-6} For any $s \in S$, we have $V_s^- \ne 0$.
  \end{enumerate}
\end{lemma}

\begin{proof}
  For \eqref{lem-R-1}, we have $s \cdot (s \cdot v- v) = v - s \cdot v$.
  Thus, $s \cdot v - v \in V_s^-$.

  Let $U := \sum_{s \in S} V_s^-$.
  For any vector $u \in U$ and any $s \in S$, we have $s \cdot u - u \in V_s^-$ by \eqref{lem-R-1}.
  Thus, $s \cdot u \in U$.
  So, $U$ stays invariant under the action of $W$.
  But $V$ is an irreducible representation of $W$.
  As a result, we must have $U = V$.
  The point \eqref{lem-R-2} is proved.

  The assertion $V_r^- \cap V_t^- = 0$ in \eqref{lem-R-3} follows from Theorem \ref{thm-A1} and the fact that $V$ is of $\af$-function value $1$.

  For \eqref{lem-R-4}, suppose $m_{rt} = 2$ and $v \in V_t^-$.
  We have
  \begin{align*}
    t \cdot (r \cdot v - v) & = tr \cdot v - t \cdot v \\
     & = rt \cdot v + v \\
     & = - r \cdot v + v.
  \end{align*}
  Thus, $r \cdot v - v \in V_t^-$.
  But it also holds that $r \cdot v - v \in V_r^-$ by \eqref{lem-R-1}.
  This forces $r \cdot v - v = 0$ because $V_r^- \cap V_t^- = 0$ as we have seen in \eqref{lem-R-3}.

  Now we prove \eqref{lem-R-5}.
  Suppose first that $m_{rt} = 3$.
  The subgroup $\langle r,t \rangle$ generated by $r$ and $t$ is isomorphic to the finite dihedral group $D_3$.
  Note that the group algebra $\bC[D_3]$ is semisimple.
  Therefore, $V$ decomposes as a representation of $\langle r,t \rangle$ into a direct sum of copies of irreducible representations of $\langle r,t \rangle$, which are $\II, \varepsilon$ and $\rho_1$ (see Lemma \ref{lem-Dm-rep}).
  But $\varepsilon$ cannot occur in $V$ by Theorem \ref{thm-A1}.
  Thus, we may write
  \begin{equation*}
    (V, \rho) =  \II^{\oplus I}  \bigoplus \rho_1^{\oplus J}
  \end{equation*}
  as a representation of $\langle r,t \rangle$, where $I$ and $J$ are index sets.
  In view of Remark \ref{rmk-Dm-rep}, we have clearly $\dim V_r^- = \dim V_t^- = \abs{J}$.

  Suppose otherwise $m_{rt} = 2$.
  Notice that the Coxeter graph $G$ is assumed to be connected.
  We have a path in $G$ connecting $r$ and $t$, say,
  \begin{equation*}
    (r = s_0, s_1, \dots, s_{k-1}, s_k = t)
  \end{equation*}
  where $s_i \in S$ and $m_{s_i s_{i+1}} = 3$ for any $i$.
  Then, we have $\dim V_r^- = \dim V_{s_1}^- = \dots = \dim V_t^-$ as desired in \eqref{lem-R-5}.

  As a result of \eqref{lem-R-5}, if $V_s^- = 0$ for some $s \in S$, then $V_s^- = 0$ for any $s \in S$, and hence the action of $s$ on $V$ is trivial.
  But then the $\af$-function value of $V$ is zero (see Example \ref{eg-a=0}), which contradicts our assumption on $V$.
  This proves \eqref{lem-R-6}.
\end{proof}

For two arbitrary elements $r, t \in S$ such that $m_{rt} = 3$, we define a linear map $f_{tr}$ as follows,
\begin{align*}
  f_{tr} : V_r^- & \to V_t^-,  \\
   v & \mapsto  t \cdot v - v.
\end{align*}
Note that $t \cdot v - v \in V_t^-$ by Lemma \ref{lem-R}\eqref{lem-R-1}.
Thus, $f_{tr}: V_r^- \to V_t^-$ is a well-defined linear map.
Intuitively, the vector $v$ generates in $V$ a two-dimensional subrepresentation of the dihedral subgroup $\langle r,t \rangle$, and the subrepresentation is isomorphic to $\rho_1$.
In this subrepresentation the vector $v$ and $f_{tr} (v)$ play the same roles as $\beta_r$ and $\beta_t$ in Subsection \ref{subsec-dih}, as illustrated in Figure \ref{fig-ftr}.
\begin{figure}[ht]
    \centering
    \begin{tikzpicture}
      \draw[dashed] (-1.3,0)--(2,0);
      \draw[dashed] (240:1.3)--(60:2);
      \draw[->] (0,0)--(0,1.5);
      \draw[->] (0,0)--(330:1.5);
      \draw[->] (0,0)--(30:1.5);
      \draw (-0.3,0) arc (180:240:0.3);
      \draw[<->] ($(1.9,0) + (300:0.2)$) arc (300:420:0.2);
      \draw[<->] ($(60:1.9) + (0.2,0)$) arc (0:120:0.2);
      \node[left] (ar) at (0,1.5) {$v$};
      \node[below] (at) at (2.1,-0.8) {${f_{tr} (v) = t \cdot v - v}$};
      \node[right] (r) at (2.1,0) {$r$};
      \node[right] (t) at (63:2.3) {$t$};
      \node[right] (tv) at (33:1.5) {$t \cdot v$};
      \node[left] (angle) at (-0.2,-0.3) {$\frac{\uppi}{3}$};
    \end{tikzpicture}
    \caption{The vectors $v$ and $f_{tr} (v)$}\label{fig-ftr}
\end{figure}

\begin{lemma} \label{lem-fst-isom}
  For $r,t \in S$ with $m_{rt} = 3$, the map $f_{tr}$ is a linear isomorphism with inverse $f_{rt}$.
\end{lemma}

\begin{proof}
  For any $v \in V_r^-$, we have $rt \cdot v + v - t\cdot v \in V_r^- \cap V_t^-$ because
  \begin{align*}
    r \cdot (rt \cdot v + v - t\cdot v) & = t \cdot v + r \cdot v - rt \cdot v \\
     & =  t \cdot v - v - rt \cdot v \\
     & = - (rt \cdot v + v - t\cdot v), \text{ and} \\
    t \cdot (rt \cdot v + v - t\cdot v) & = trt \cdot v + t \cdot v - v \\
    & = rtr \cdot v + t \cdot v - v \\
    & = -rt \cdot v + t \cdot v - v \\
    & = - (rt \cdot v + v - t\cdot v).
  \end{align*}
  But $V_r^- \cap V_t^- = 0$ by Lemma \ref{lem-R}\eqref{lem-R-3}.
  Thus,
  \begin{equation}    \label{eq-4-3}
    rt \cdot v + v - t\cdot v = 0.
  \end{equation}
  We have then
  \begin{align*}
    f_{rt} f_{tr} (v) & = f_{rt} (t \cdot v - v) \\
     & = r \cdot (t\cdot v - v) - (t \cdot v - v) \\
     & = r t \cdot v + v - t \cdot v + v \\
    & = v \text{ (by Equation \eqref{eq-4-3})}.
  \end{align*}
  We interchange the two letters $r$ and $t$ and then we have $f_{tr} f_{rt} (v) = v$  for any $v \in V_t^-$.
\end{proof}

We have seen from Lemma \ref{lem-R}\eqref{lem-R-2} and \eqref{lem-R-3} that the representation $(V,\rho)$ satisfies the conditions \eqref{def-R-1} and \eqref{def-R-2} of Definition \ref{def-R}.
It remains to prove the condition \eqref{def-R-3} of Definition \ref{def-R}.
We split the proof into two cases depending on the shape of the Coxeter graph $G$.

\subsubsection{Case one: \texorpdfstring{$G$}{G} is a tree}

In this case, we fix arbitrarily an element $s_0 \in S$ and a nonzero vector $\alpha_{s_0} \in V_{s_0}^-$.
Such $\alpha_{s_0}$ exists because $V_{s_0}^- \ne 0$ by Lemma \ref{lem-R}\eqref{lem-R-6}.

For any vertex in $G$, say, $s \in S$, there is a unique path (without repetition of vertices) connecting $s_0$ and $s$ since $G$ is a tree.
Suppose this path is
\begin{equation*}
    (s_0, s_1, \dots, s_{k-1}, s_k = s)
\end{equation*}
where $s_i \in S$ and $m_{s_i s_{i+1}} = 3$ for any $i$.
We define
\begin{equation*}
  \alpha_s := f_{s_k s_{k-1}} f_{s_{k-1} s_{k-2}} \cdots f_{s_1 s_0} (\alpha_{s_0}).
\end{equation*}
Then, $\alpha_s \in V_s^-$, and $\alpha_s$ is nonzero by Lemma \ref{lem-fst-isom}.
In particular, if $s = s_0$, then the vector is $\alpha_{s_0}$ we have chosen.

\begin{lemma} \label{lem-tree-subrep} \leavevmode
  \begin{enumerate}
    \item \label{lem-tree-subrep-1} Suppose $s, t \in S$ such that $m_{st} = 3$.
        Then, we have $t \cdot \alpha_s = \alpha_s + \alpha_t$.
    \item \label{lem-tree-subrep-2} The subspace $U := \sum_{s \in S} \bC \alpha_s$ spanned by $\{\alpha_s \mid s \in S\}$ is a subrepresentation of $W$ in $V$.
  \end{enumerate}
\end{lemma}

\begin{proof}
  Suppose $m_{st} = 3$, and suppose the path in $G$ connecting $s_0$ and $s$ is
  \begin{equation*}
    (s_0, s_1, \dots, s_{k-1}, s).
  \end{equation*}
  If $s_{k-1} = t$, then we have
  \begin{align*}
    t \cdot \alpha_s & = (t \cdot \alpha_s - \alpha_s) + \alpha_s \\
     & = f_{ts} (\alpha_s) + \alpha_s \\
     & = f_{ts} (f_{st} \alpha_{t}) + \alpha_s \text{ (by definition of $\alpha_s$ and $\alpha_t$)} \\
    & = \alpha_t + \alpha_s \text{ (by Lemma \ref{lem-fst-isom})}.
  \end{align*}
  If $s_{k-1} \ne t$, then the path in $G$ connecting $s_0$ and $t$ is exactly
  \begin{equation*}
    (s_0, s_1, \dots, s_{k-1}, s,  t).
  \end{equation*}
  Then, we have
  \begin{align*}
    t \cdot \alpha_s & = (t \cdot \alpha_s - \alpha_s) + \alpha_s \\
     & = f_{ts} (\alpha_s) + \alpha_s \\
    & = \alpha_t + \alpha_s \text{ (by the definition of $\alpha_t$)}.
  \end{align*}
  Therefore, we always have $t \cdot \alpha_s = \alpha_t + \alpha_s$.
  The point \eqref{lem-tree-subrep-1} is proved.

  For \eqref{lem-tree-subrep-2}, it suffices to show that $t \cdot \alpha_s \in U$ for any $s,t \in S$.
  If $t = s$ then this is clear since $s \cdot  \alpha_s = - \alpha_s \in U$.
  If $m_{st} = 2$ then $t \cdot \alpha_s = \alpha_s \in U$ by Lemma \ref{lem-R}\eqref{lem-R-4}.
  If $m_{st} = 3$ then $t \cdot \alpha_s = \alpha_t + \alpha_s \in U$ by \eqref{lem-tree-subrep-1}.
  Thus, $U$ is a subrepresentation.
\end{proof}

\begin{corollary} \label{cor-tree}
  $\dim V_s^- = 1$ for any $s \in S$.
\end{corollary}

\begin{proof}
  Since $V$ is irreducible, we have $V = \sum_{s \in S} \bC \alpha_s$ by Lemma \ref{lem-tree-subrep}\eqref{lem-tree-subrep-2}.
  Notice that $\alpha_s \in V_s^-$ is a nonzero vector.
  So $\dim V_s^- \ge 1$.

  For arbitrarily fixed $s \in S$, we may suppose $s_1, \dots, s_l$ are elements in $S$ such that $\{\alpha_s, \alpha_{s_1}, \dots, \alpha_{s_l}\}$ is a basis of $V$.
  Suppose $v \in V_s^-$.
  Then, we can write
  \begin{equation}   \label{eq-cor-tree-1}
    v = c_s \alpha_s + \sum_{1 \le i \le l} c_i \alpha_{s_i},
  \end{equation}
  where $c_s, c_{i} \in \bC$.
  By Lemma \ref{lem-R}\eqref{lem-R-4} and Lemma \ref{lem-tree-subrep}\eqref{lem-tree-subrep-1}, we have
  \begin{equation}   \label{eq-cor-tree-2}
   \begin{split}
    -v & = s \cdot v \\
     & = - c_s \alpha_s + \sum_{\substack{1 \le i \le l \\ m_{ss_i} = 2}} c_i \alpha_{s_i} + \sum_{\substack{1 \le i \le l \\ m_{ss_i} = 3}} c_i (\alpha_{s_i} + \alpha_s) \\
     & = \Bigl(- c_s + \sum_{\substack{1 \le i \le l \\ m_{ss_i} = 3}} c_i \Bigr) \alpha_s + \sum_{1 \le i \le l} c_i \alpha_{s_i}.
    \end{split}
  \end{equation}
  Compare the equations \eqref{eq-cor-tree-1} and \eqref{eq-cor-tree-2}, and notice that $\{\alpha_s, \alpha_{s_1}, \dots, \alpha_{s_l}\}$ is a basis of $V$.
  Then, we must have $c_1 = \dots = c_l = 0$.
  Thus, $v \in \bC \alpha_s$, and then $V_s^- = \bC \alpha_s$ is of dimension $1$.
\end{proof}

In conclusion, in the case where $G$ is a tree, the representation $V$ is an R-representation.

\subsubsection{Case two: \texorpdfstring{$G$}{G} has one cycle} \label{subsubsec-cycle}

In this case, $G$ has (only) one cycle.
We assume $\{s_{i} \in S \mid 0 \le i \le n\}$ is the set of vertices on that cycle such that $m_{s_{i} s_{i+1}} = 3$ for any $i \le n-1$ and $m_{s_0s_n} = 3$, as illustrated in Figure \ref{fig-cycle}.
\begin{figure} [ht]
  \centering
  \begin{tikzpicture}
    \node [circle, draw, inner sep=2pt, label=below:$s_0$] (s0) at (0,0) {};
    \node [circle, draw, inner sep=2pt, label=below:$s_1$] (s1) at (-1,0) {};
    \node [circle, draw, inner sep=2pt, label=below:$s_n$] (sn) at (1,0) {};
    \node [circle, draw, inner sep=2pt, label=below:$s_{n-1}$] (sm) at (2,0) {};
    \node [circle, draw, inner sep=2pt, label=left:$s_2$] (s2) at (-1,1) {};
    \node [circle, draw, inner sep=2pt, label=right:$s_{n-2}$] (sm2) at (2,1) {};
    \draw (-0.3,1) -- (s2) -- (s1) -- (s0) -- (sn) -- (sm) -- (sm2) -- (1.3,1);
    \node (d) at (0.5, 1) {$\cdots$};
  \end{tikzpicture}
  \caption{The cycle in $G$}\label{fig-cycle}
\end{figure}

For any vector $v \in V_{s_0}^-$, we define
\begin{equation*}
  X (v) : = f_{s_0 s_n} f_{s_n s_{n-1}} \cdots f_{s_2 s_1} f_{s_1 s_0} (v).
\end{equation*}
Then, by Lemma \ref{lem-fst-isom}, $X : V_{s_0}^- \to V_{s_0}^-$ is a linear isomorphism of the vector space $V_{s_0}^-$.
Thus, we may regard $V_{s_0}^-$ as a module of the Laurent polynomial ring $\bC[X^{\pm 1}]$.

Suppose $U_{s_0} \subseteq V_{s_0}^-$ is a submodule of $\bC[X^{\pm 1}]$.
For any $t \in S$, we define a subspace $U_t \subseteq V_t^-$ as follows.
We choose a path (without repetition of vertices) in $G$ connecting $s_0$ and $t$, say,
\begin{equation*}
  (s_0 = t_0, t_1, t_2, \dots,t_{k-1}, t_k = t),
\end{equation*}
where $t_1, \dots, t_{k-1} \in S$ such that $m_{t_i t_{i+1}} = 3$ for any $i$.
We define
\begin{equation*}
  U_t := f_{t_k t_{k-1}} f_{t_{k-1} t_{k-2}} \dots f_{t_1 t_0} (U_{s_0}).
\end{equation*}
In particular, if $t = s_0$, then the subspace defined is just $U_{s_0}$.
Note that the path connecting $s_0$ and $t$ is not unique in general.
But the point \eqref{lem-cycle-wd-1} of the following lemma ensures that $U_t$ is well defined.

\begin{lemma} \label{lem-cycle-wd} \leavevmode
  \begin{enumerate}
    \item \label{lem-cycle-wd-1} The definition of $U_t$ is independent of the choice of the path connecting $s_0$ and $t$.
    \item \label{lem-cycle-wd-2} If $r,t \in S$ such that $m_{rt} =3$, then $f_{rt} (U_t) = U_r$.
  \end{enumerate}
\end{lemma}

\begin{proof}
  We prove \eqref{lem-cycle-wd-1} first.
  If there is only one path connecting $s_0$ and $t$, then there is nothing to prove.
  If $t = s_i$ for some $i \le n$, then there are two paths connecting $s_0$ and $s_i$, namely,
  \begin{equation}   {\label{eq-lem-wd-1}}
    (s_0, s_1, s_2, \dots, s_{i-1}, s_i) \text{ and } (s_0, s_n, s_{n-1}, \dots, s_{i+1}, s_i).
  \end{equation}
  Recall that $U_{s_0} \subseteq V_{s_0}^-$ is a submodule of $\bC[X^{\pm 1}]$.
  Thus, $X(U_{s_0}) = U_{s_0}$, that is,
  \begin{equation*}
    f_{s_0 s_n} f_{s_n s_{n-1}} \cdots f_{s_2 s_1} f_{s_1 s_0} (U_{s_0}) = U_{s_0}.
  \end{equation*}
  By Lemma \ref{lem-fst-isom}, we have then
  \begin{equation}   \label{eq-lem-wd-2}
    f_{s_i s_{i-1}} \cdots f_{s_2 s_1} f_{s_1 s_0} (U_{s_0}) = f_{s_i s_{i+1}} \cdots f_{s_{n-1} s_n} f_{s_n s_0} (U_{s_0})
  \end{equation}
  Notice that the two paths in \eqref{eq-lem-wd-1} are the only paths connecting $s_0$ and $s_i$ because there is only one cycle in $G$.
  Thus, the equality \eqref{eq-lem-wd-2} implies that the subspace $U_{s_i}$ is independent of the choice of the path.

  Suppose otherwise $t$ is not on the cycle and there are two (and only two) paths connecting $s_0$ and $t$.
  Then, there exists a unique index $i$ with $1 \le i \le n$ such that there is a unique path
  \begin{equation*}
    (s_i = r_0, r_1, \dots, r_l = t)
  \end{equation*}
  connecting $s_i$ and $t$, and $r_1, \dots, r_l$ do not lie on the cycle in $G$, as illustrated in Figure \ref{fig-path}.
  Then, by Equation \eqref{eq-lem-wd-2} we have
  \begin{align*}
     & \mathrel{\phantom{=}} f_{r_l r_{l-1}} \cdots f_{r_2 r_1} f_{r_1 s_i} f_{s_i s_{i-1}} \cdots f_{s_2 s_1} f_{s_1 s_0} (U_{s_0})
    \\ & = f_{r_l r_{l-1}} \cdots f_{r_2 r_1} f_{r_1 s_i} f_{s_i s_{i+1}} \cdots f_{s_{n-1} s_n} f_{s_n s_0} (U_{s_0}).
  \end{align*}
  This indicates that the definition of  $U_t$ is independent of the choice of the path.
  The point \eqref{lem-cycle-wd-1} is proved.
  \begin{figure} [ht]
  \centering
  \begin{tikzpicture}
    \node [circle, draw, inner sep=2pt, label=below:$s_0$] (s0) at (0,0) {};
    \node [circle, draw, inner sep=2pt, label=below:$s_1$] (s1) at (-1,0) {};
    \node [circle, draw, inner sep=2pt, label=below:$s_n$] (sn) at (1,0) {};
    \node (sm) at (1.8,0.25) {$\iddots$};
    \node [circle, draw, inner sep=2pt, label=left:$s_2$] (s2) at (-1,1) {};
    \node [circle, draw, inner sep=2pt, label=225:$s_i$] (sm2) at (2,1) {};
    \draw (-0.3,1) -- (s2) -- (s1) -- (s0) -- (sn) -- (1.5,0);
    \draw (2,0.5) -- (sm2) -- (1.3,1);
    \node (d) at (0.5, 1) {$\cdots$};
    \node [circle, draw, inner sep=2pt, label=below:$r_1$] (r1) at (3,1) {};
    \node [circle, draw, inner sep=2pt, label=below:$r_{l-1}$] (rl) at (4.5,1) {};
    \node [circle, draw, inner sep=2pt] (t) at (5.5,1) {};
    \node [label=below:${r_l = t}$] (rt) at (5.8,1) {};
    \draw (sm2) -- (r1) -- (3.4, 1);
    \draw (t) -- (rl) -- (4.1,1);
    \node (dd) at (3.75,1) {$\cdots$};
  \end{tikzpicture}
  \caption{The path connecting $s_i$ and $t$}\label{fig-path}
  \end{figure}

  For \eqref{lem-cycle-wd-2}, suppose first that there is a path connecting $s_0$ and $t$ such that $r$ does not lie in this path.
  Then, by definition we have $f_{rt} (U_t) = U_r$.
  Suppose otherwise that a path connecting $s_0$ and $t$ passes through $r$, then by definition again we have $f_{tr} (U_r) = U_t$.
  By applying $f_{rt}$ on both sides, and using Lemma \ref{lem-fst-isom}, we obtain $U_r = f_{rt} (U_t)$.
\end{proof}

\begin{lemma}\label{lem-cycle}
  Suppose $U_{s_0} \subseteq V_{s_0}^-$ is a submodule of $\bC[X^{\pm 1}]$.
  Let $U := \sum_{t \in S} U_t \subseteq V$ where $U_t$ is defined as above.
  Then, we have:
  \begin{enumerate}
    \item \label{lem-cycle-1} $U$ is a subrepresentation of $W$ in $V$.
    \item \label{lem-cycle-2} $U \cap V_{s_0}^- = U_{s_0}$.
  \end{enumerate}
\end{lemma}

\begin{proof}
  The proof is similar to that of Lemma \ref{lem-tree-subrep}\eqref{lem-tree-subrep-2} and Corollary  \ref{cor-tree}.

  To prove \eqref{lem-cycle-1}, it suffices to show $s \cdot v \in U$ for any $s,t \in S$ and $v \in U_t$.
  Note that $U_t \subseteq V_t^-$.
  Thus, if $s = t$ or if $m_{st} = 2$, then the assertion is clear by the definition of $V_t^-$ and Lemma \ref{lem-R}\eqref{lem-R-4}.
  Now suppose $m_{st} = 3$.
  Then, we have
  \begin{equation*}
    s \cdot v = (s \cdot v - v) + v = f_{st} (v) + v.
  \end{equation*}
  By Lemma \ref{lem-cycle-wd}\eqref{lem-cycle-wd-2}, we have $f_{st} (v) + v \in U_s + U_t$.
  Thus, $s \cdot v \in U$.

  For \eqref{lem-cycle-2}, we choose a basis $\{\alpha_1, \dots, \alpha_k\}$ of $U_{s_0}$, and suppose
  \[\{\alpha_1, \dots, \alpha_k, \alpha_{k+1}, \dots, \alpha_l\}\]
  is a basis of $U$ such that for any index $i \ge k+1$, the vector $\alpha_i$ belongs to $U_t$ for some $t \in S\setminus \{s_0\}$.
  For such index $i$ and such element $t$, if $m_{ts_0} = 2$, then we have $s_0 \cdot \alpha_i = \alpha_i$ by Lemma \ref{lem-R}\eqref{lem-R-4}; if $m_{ts_0} = 3$, then we have
  \begin{equation*}
    s_0 \cdot \alpha_i = (s_0 \cdot \alpha_i -\alpha_i) + \alpha_i = f_{s_0 t} (\alpha_i) + \alpha_i.
  \end{equation*}
  Note that $f_{s_0 t} (\alpha_i) \in U_{s_0}$ by Lemma \ref{lem-cycle-wd}\eqref{lem-cycle-wd-2}.
  Thus, $f_{s_0 t} (\alpha_i)$ can be written as a linear combination of $\alpha_1, \dots, \alpha_k$.
  Suppose now $v \in U \cap V_{s_0}^-$.
  We write
  \begin{equation*}
    v = \sum_{1 \le i \le l} c_i \alpha_i, \quad c_i \in \bC.
  \end{equation*}
  Then
  \begin{align*}
    -v & = s_0 \cdot v \\
     & = -\sum_{1 \le i \le k} c_i \alpha_i + \sum_{k+1 \le i \le l} c_i \bigl(f_{s_0 t} (\alpha_i) + \alpha_i\bigr) \\
     & = \text{(some linear combination of $\alpha_1, \dots, \alpha_k$) } + \sum_{k+1 \le i \le l} c_i \alpha_i.
  \end{align*}
  Thus, we must have $c_{k+1} = \dots = c_l = 0$, and $v \in U_{s_0}$.
  This proves $U \cap V_{s_0}^- \subseteq U_{s_0}$.
  The inclusion $U \cap V_{s_0}^- \supseteq U_{s_0}$ is obvious.
\end{proof}

\begin{corollary} \label{cor-cycle}
  $\dim V_{s_0}^- = 1$.
\end{corollary}

\begin{proof}
  By Hilbert's Nullstellensatz, any simple module of $\bC[X^{\pm 1}]$ is one dimensional.
  Therefore, if $\dim V_{s_0}^- > 1$, then there exists a proper $\bC[X^{\pm 1}]$-submodule in $V_{s_0}^-$, say, $U_{s_0} \subsetneq V_{s_0}^-$.
  Let $U \subseteq V$ be the subrepresentation of $W$ defined as in Lemma \ref{lem-cycle}.
  Then, Lemma \ref{lem-cycle} implies that $U \subsetneq V$ is a proper subrepresentation of $W$, contradicting the irreducibility of $V$.
  Thus, $\dim V_{s_0}^- = 1$.
\end{proof}


By Corollary \ref{cor-cycle} and Lemma \ref{lem-R}\eqref{lem-R-5}, we have $\dim V_s^- = 1$ for any $s \in S$.
The proof of Theorem \ref{thm-A1=R} is complete.

\subsection{Appendix: The construction of irreducible R-representations} \label{subsec-R}

As promised in Remark \ref{rmk-simply-laced}\eqref{rmk-simply-laced-1}, we present in this subsection the explicit construction in \cite{Hu23} of the irreducible R-representations in the cases where the simply laced Coxeter graph has no cycle or only one cycle.

If the Coxeter graph $G$ is a simply laced tree, then the only irreducible representation of $W$ of $\af$-function value $1$ is the geometric representation $V_\text{geom}$ (if $V_\text{geom}$ is irreducible) or the simple quotient of $V_\text{geom}$ (if  $V_\text{geom}$ is reducible).
Recall that if $V_\text{geom}$ is reducible, then $V_\text{geom}$ has a maximal subrepresentation $V_0$ on which the $W$-action is trivial, and the quotient $V_\text{geom} / V_0$ is irreducible.
See \cite[Ch.\ V, no.\ 7]{Bourbaki-Lie456} for more details.

Suppose now there is one and only one cycle in the Coxeter graph $G$ of an irreducible simply laced Coxeter group $(W,S)$.
Suppose moreover that $s_0, s_1, \dots, s_n$ are the vertices on the cycle in $G$ as in the sub-sub-section \ref{subsubsec-cycle}.
For any number $x \in \bC^\times$, we define a representation $(\tld{V}_x, \tld{\rho}_x)$ of $W$ as follows.
As a vector space, $\tld{V}_x := \bigoplus_{s \in S} \bC \alpha_s$ is spanned by a formal basis $\{\alpha_s \mid s \in S\}$.
The action $\tld{\rho}_x$ of $W$ on $\tld{V}_x$ is defined by:
\begin{enumerate}
  \item $s \cdot \alpha_s = - \alpha_s$, for any $s \in S$;
  \item $s_0 \cdot \alpha_{s_n} := \alpha_{s_n} + x \alpha_{s_0}$, and $s_n \cdot \alpha_{s_0} := \alpha_{s_0} + \frac{1}{x} \alpha_{s_n}$;
  \item for any other pairs $s,t \in S$ with $s \ne t$, that is, $\{s,t\} \ne \{s_0, s_n\}$, we define $s \cdot \alpha_t := \alpha_t$ if $m_{st} = 2$, and $s \cdot \alpha_t := \alpha_t + \alpha_s$ if $m_{st} = 3$.
\end{enumerate}
It can be verified that $(\tld{V}_x, \tld{\rho}_x)$ is a well defined representation of $W$.
Let $U_x := \{v \in \tld{V}_x \mid w \cdot v = v, \forall w \in W\}$ be the subrepresentation with trivial group action ($U_x$ is usually zero, but it is possible to be nonzero, e.g., when $W$ is of affine type $\tld{\sfA}_n$ and $x = 1$).
We define $(V_x, \rho_x) := (\tld{V}_x, \tld{\rho}_x) / U_x$ to be the quotient representation.
It turns out that the representations $\{(V_x, \rho_x) \mid x \in \bC^\times\}$ are pairwise non-isomorphic, and they are all the irreducible R-representations of $W$, or equivalently, all the irreducible representations of $W$ of $\af$-function value $1$  (up to isomorphism).
See \cite{Hu23} for more details of the construction.


\bibliographystyle{amsplain}
\bibliography{a-fun-one}

@book {BB05,
    AUTHOR = {Bj\"{o}rner, Anders and Brenti, Francesco},
     TITLE = {Combinatorics of {C}oxeter {G}roups},
    SERIES = {Graduate Texts in Mathematics},
    VOLUME = {231},
 PUBLISHER = {Springer},
   address = {New York},
      YEAR = {2005},
     PAGES = {xiv+363}}

@article {BCNS15,
    AUTHOR = {Bugaenko, Vadim and Cherniavsky, Yonah and Nagnibeda, Tatiana
              and Shwartz, Robert},
     TITLE = {Weighted {C}oxeter graphs and generalized geometric
              representations of {C}oxeter groups},
   JOURNAL = {Discrete Appl. Math.},
  FJOURNAL = {Discrete Applied Mathematics. The Journal of Combinatorial
              Algorithms, Informatics and Computational Sciences},
    VOLUME = {192},
      YEAR = {2015},
     PAGES = {17--27}}

@article {Belolipetsky04,
    AUTHOR = {Belolipetsky, Mikhail},
     TITLE = {Cells and representations of right-angled {C}oxeter groups},
   JOURNAL = {Selecta Math. (N.S.)},
  FJOURNAL = {Selecta Mathematica. New Series},
    VOLUME = {10},
      YEAR = {2004},
    NUMBER = {3},
     PAGES = {325--339}}

@book {Bourbaki-Lie456,
    AUTHOR = {Bourbaki, Nicolas},
    TITLE = {{Lie Groups and Lie Algebras. Chapters 4--6}},
    SERIES = {Elements of Mathematics},
    NOTE = {Translated from the 1968 French original by Andrew Pressley},
    PUBLISHER = {Springer-Verlag},
    address={Berlin},
    YEAR = {2002},
    PAGES = {xii+300}}

@book {Bonnafe,
    AUTHOR = {Bonnaf\'{e}, C\'{e}dric},
     TITLE = {Kazhdan-{L}usztig cells with unequal parameters},
    SERIES = {Algebra and Applications},
    VOLUME = {24},
 PUBLISHER = {Springer},
   address = {Cham},
      YEAR = {2017},
     PAGES = {xxv+348}}

@unpublished{Chen25,
    author = {Chen, Xiaoyu and Hu, Hongsheng},
    title = {Boundedness of {L}usztig's $\boldsymbol{a}$-function for {C}oxeter groups of finite rank},
    note = {preprint, arXiv:2503.06432},
    year = {2025}}

@article {Donnelly11,
    AUTHOR = {Donnelly, Robert G.},
     TITLE = {Root systems for asymmetric geometric representations of
              {C}oxeter groups},
   JOURNAL = {Comm. Algebra},
  FJOURNAL = {Communications in Algebra},
    VOLUME = {39},
      YEAR = {2011},
    NUMBER = {4},
     PAGES = {1298--1314}}

@article {Douglass90,
    AUTHOR = {Douglass, J. Matthew},
     TITLE = {Cells and the reflection representation of {W}eyl groups and
              {H}ecke algebras},
   JOURNAL = {Trans. Amer. Math. Soc.},
  FJOURNAL = {Transactions of the American Mathematical Society},
    VOLUME = {318},
      YEAR = {1990},
    NUMBER = {1},
     PAGES = {373--399}}

@article {DPWX22,
    AUTHOR = {Dimitrov, Ivan and Paquette, Charles and Wehlau, David and Xu,
              Tianyuan},
     TITLE = {Subregular {$J$}-rings of {C}oxeter systems via quiver path
              algebras},
   JOURNAL = {J. Algebra},
  FJOURNAL = {Journal of Algebra},
    VOLUME = {612},
      YEAR = {2022},
     PAGES = {526--576}
}

@article {EW14,
    AUTHOR = {Elias, Ben and Williamson, Geordie},
    TITLE = {The {H}odge theory of {S}oergel bimodules},
    JOURNAL = {Ann. of Math. (2)},
    FJOURNAL = {Annals of Mathematics. Second Series},
    VOLUME = {180},
    YEAR = {2014},
    NUMBER = {3},
    PAGES = {1089--1136}}

@article {Hee91,
    AUTHOR = {H\'{e}e, Jean-Yves},
     TITLE = {Syst\`eme de racines sur un anneau commutatif totalement
              ordonn\'{e}},
   JOURNAL = {Geom. Dedicata},
  FJOURNAL = {Geometriae Dedicata},
    VOLUME = {37},
      YEAR = {1991},
    NUMBER = {1},
     PAGES = {65--102}}

@article {Hu22,
    AUTHOR = {Hu, Hongsheng},
     TITLE = {Some Infinite-Dimensional Representations of Certain
              {C}oxeter Groups},
   JOURNAL = {Ann. Comb.},
  FJOURNAL = {Annals of Combinatorics},
    VOLUME = {29},
      YEAR = {2025},
    NUMBER = {1},
     PAGES = {101--115}}

@article {Hu23,
    AUTHOR = {Hu, Hongsheng},
     TITLE = {Reflection representations of {C}oxeter groups and homology of
              {C}oxeter graphs},
   JOURNAL = {Algebr. Represent. Theory},
  FJOURNAL = {Algebras and Representation Theory},
    VOLUME = {27},
      YEAR = {2024},
    NUMBER = {1},
     PAGES = {961--994}}

@article {Krammer09,
    AUTHOR = {Krammer, Daan},
     TITLE = {The conjugacy problem for {C}oxeter groups},
   JOURNAL = {Groups Geom. Dyn.},
  FJOURNAL = {Groups, Geometry, and Dynamics},
    VOLUME = {3},
      YEAR = {2009},
    NUMBER = {1},
     PAGES = {71--171}}

@article {KL79,
    AUTHOR = {Kazhdan, David and Lusztig, George},
    TITLE = {Representations of {C}oxeter groups and {H}ecke algebras},
    JOURNAL = {Invent. Math.},
    FJOURNAL = {Inventiones Mathematicae},
    VOLUME = {53},
    YEAR = {1979},
    NUMBER = {2},
    PAGES = {165--184}}

@article {LS19,
    AUTHOR = {Li, Yan and Shi, Jian-Yi},
     TITLE = {The boundedness of a weighted {C}oxeter group with
              non-3-edge-labeling graph},
   JOURNAL = {J. Algebra Appl.},
  FJOURNAL = {Journal of Algebra and its Applications},
    VOLUME = {18},
      YEAR = {2019},
    NUMBER = {5},
     PAGES = {1950085, 43}}

@article {Lusztig83-intrep,
    AUTHOR = {Lusztig, George},
    TITLE = {Some examples of square integrable representations of semisimple {$p$}-adic groups},
   JOURNAL = {Trans. Amer. Math. Soc.},
   FJOURNAL = {Transactions of the American Mathematical Society},
   VOLUME = {277},
   YEAR = {1983},
   NUMBER = {2},
   PAGES = {623--653}}

@incollection {Lusztig85-cell-i,
    AUTHOR = {Lusztig, George},
    TITLE = {Cells in affine {W}eyl groups},
    BOOKTITLE = {Algebraic {G}roups and {R}elated {T}opics ({K}yoto/{N}agoya, 1983)},
    editor={Hotta, R.},
    SERIES = {Adv. Stud. Pure Math.},
    VOLUME = {6},
    PAGES = {255--287},
    PUBLISHER = {North-Holland},
    address={Amsterdam},
    YEAR = {1985}}

@article {Lusztig87-cell-ii,
    AUTHOR = {Lusztig, George},
    TITLE = {Cells in affine {W}eyl groups. {II}},
    JOURNAL = {J. Algebra},
    FJOURNAL = {Journal of Algebra},
    VOLUME = {109},
    YEAR = {1987},
    NUMBER = {2},
    PAGES = {536--548}}

@article {Lusztig87-cell-iii,
    AUTHOR = {Lusztig, George},
     TITLE = {Cells in affine {W}eyl groups. {III}},
   JOURNAL = {J. Fac. Sci. Univ. Tokyo Sect. IA Math.},
  FJOURNAL = {Journal of the Faculty of Science. University of Tokyo.
              Section IA. Mathematics},
    VOLUME = {34},
      YEAR = {1987},
    NUMBER = {2},
     PAGES = {223--243}}

@unpublished{Lusztig14-hecke-unequal,
    author={Lusztig, George},
    title={Hecke algebras with unequal parameters},
    note={{arXiv}:0208154v2. Revised version of \emph{Hecke Algebras with Unequal Parameters}, 2003, CRM Monograph Series 18, Amer. Math. Soc., Providence, RI},
    year={2014}}

@article {Matsumoto64,
    AUTHOR = {Matsumoto, Hideya},
     TITLE = {G\'{e}n\'{e}rateurs et relations des groupes de {W}eyl g\'{e}n\'{e}ralis\'{e}s},
   JOURNAL = {C. R. Acad. Sci. Paris},
  FJOURNAL = {Comptes Rendus Hebdomadaires des S\'{e}ances de l'Acad\'{e}mie des
              Sciences},
    VOLUME = {258},
      YEAR = {1964},
     PAGES = {3419--3422},
      ISSN = {0001-4036}}

@book {Serre77,
    AUTHOR = {Serre, Jean-Pierre},
     TITLE = {Linear representations of finite groups},
    SERIES = {Graduate Texts in Mathematics, Vol. 42},
      NOTE = {Translated from the second French edition by Leonard L. Scott},
 PUBLISHER = {Springer-Verlag, New York-Heidelberg},
      YEAR = {1977},
     PAGES = {x+170}}

@article {Shi87,
    AUTHOR = {Shi, Jian-Yi},
     TITLE = {A two-sided cell in an affine {W}eyl group},
   JOURNAL = {J. London Math. Soc. (2)},
  FJOURNAL = {Journal of the London Mathematical Society. Second Series},
    VOLUME = {36},
      YEAR = {1987},
    NUMBER = {3},
     PAGES = {407--420}}

@incollection {Tits69,
    AUTHOR = {Tits, Jacques},
     TITLE = {Le probl\`eme des mots dans les groupes de {C}oxeter},
 BOOKTITLE = {Symposia {M}athematica ({INDAM}, {R}ome, 1967/68), {V}ol. 1},
     PAGES = {175--185},
 PUBLISHER = {Academic Press, London},
      YEAR = {1969}}

@Article{Vinberg71,
 Author = {Vinberg, \`Ernest Borisovich},
 Title = {Discrete linear groups generated by reflections},
 FJournal = {Mathematics of the USSR. Izvestiya},
 Journal = {Math. USSR, Izv.},
 ISSN = {0025-5726},
 Volume = {5},
 Pages = {1083--1119},
 Year = {1971}}

@book {Xi94,
    AUTHOR = {Xi, Nanhua},
     TITLE = {Representations of affine {H}ecke algebras},
    SERIES = {Lecture Notes in Mathematics},
    VOLUME = {1587},
 PUBLISHER = {Springer-Verlag, Berlin},
      YEAR = {1994},
     PAGES = {viii+137}}

@article {Xi12,
    AUTHOR = {Xi, Nanhua},
    TITLE = {Lusztig's {$A$}-function for {C}oxeter groups with complete graphs},
    JOURNAL = {Bull. Inst. Math. Acad. Sin. (N.S.)},
    FJOURNAL = {Bulletin of the Institute of Mathematics. Academia Sinica. New Series},
    VOLUME = {7},
    YEAR = {2012},
    NUMBER = {1},
    PAGES = {71--90}}

@article {Xie17,
    AUTHOR = {Xie, Xun},
     TITLE = {The lowest two-sided cell of a {C}oxeter group with complete
              graph},
   JOURNAL = {J. Algebra},
  FJOURNAL = {Journal of Algebra},
    VOLUME = {489},
      YEAR = {2017},
     PAGES = {38--58}}

@article {Xu19,
    AUTHOR = {Xu, Tianyuan},
     TITLE = {On the subregular {$J$}-rings of {C}oxeter systems},
   JOURNAL = {Algebr. Represent. Theory},
  FJOURNAL = {Algebras and Representation Theory},
    VOLUME = {22},
      YEAR = {2019},
    NUMBER = {6},
     PAGES = {1479--1512}}

@article {Zhou13,
    AUTHOR = {Zhou, Peipei},
     TITLE = {Lusztig's {$a$}-function for {C}oxeter groups of rank 3},
   JOURNAL = {J. Algebra},
  FJOURNAL = {Journal of Algebra},
    VOLUME = {384},
      YEAR = {2013},
     PAGES = {169--193}}

\end{document}